\documentclass[11pt,reqno]{amsart}
\usepackage{diagram}
\usepackage{amssymb}
\usepackage{cite, url, hyperref}
\usepackage{xcolor}
\usepackage[pagewise]{lineno}
\usepackage{appendix}
\numberwithin{equation}{section}

\theoremstyle{definition}
\newtheorem{thm}{Theorem}[section]
\newtheorem{cor}[thm]{Corollary}
\newtheorem{lem}[thm]{Lemma}
\newtheorem{exa}[thm]{Example}
\newtheorem{prop}[thm]{Proposition}
\newtheorem{defi}[thm]{Definition}
\newtheorem{rem}[thm]{Remark}

\DeclareMathOperator{\Pic}{\mathrm{Pic}}

\DeclareMathOperator{\pr}{\mathrm{pr}}

\DeclareMathOperator{\mo}{\mathcal{O}}
\DeclareMathOperator{\x}{\mathcal{X}}

\newcommand{\mbb}[1]{\mathbb{#1}}
\newcommand{\mb}[1]{\mathbb{#1}}
\newcommand{\mc}[1]{\mathcal{#1}}
\newcommand{\mr}[1]{\mathrm{#1}}
\newcommand{\ov}[1]{\overline{#1}}
\newcommand{\wt}[1]{\widetilde{#1}}
\newcommand{\mf}[1]{\mathfrak{#1}}

 \newcommand{\new}[1]{{\color{blue} \sf Ananyo: [#1]}}

\begin{document}

\title{Extended operational Chow group and Lefschetz $(1,1)$-theorem}

\author[A. Dan]{Ananyo Dan}

\address{CUNEF Universidad, C. de Leonardo Prieto Castro, 2, Moncloa - Aravaca, 28040 Madrid, Spain}

\email{dan.ananyo@cunef.edu}

\author[I. Kaur]{Inder Kaur}

\address{School of Mathematics \& Statistics, University of Glasgow, Glasgow G12 8QQ, U.K.}

\email{inder.kaur@glasgow.ac.uk}

\subjclass[2010]{$14$C$15$, $14$C$30$, $32$S$35$, $32$G$20$}


\date{\today}

\begin{abstract}
Let $X$ be a singular, projective variety. For every $p>0$, $H^{2p}(X,\mb{Q})$ is equipped with a mixed Hodge structure.
The elements of $\mr{Gr}^W_{2p}H^{2p}(X,\mb{Q}) \cap H^{p,p} \mr{Gr}^W_{2p}H^{2p}(X,\mb{C})$ will be called \emph{Hodge (p,p)-classes}.
The purpose of this article, is to study the Bloch-Gillet-Soul\'{e} (BGS) cycle class map from the 
$p$-th operational Chow group $A^p(X)$ to the space of $(p,p)$-Hodge classes. 
We show that if $p=1$ and $X$ is a normal surface with at worst rational singularities, then the BGS 
cycle class map is surjective. This extends the Lefschetz $(1,1)$-theorem to the setup of rational surface singularities.
However, the BGS map is not always surjective. For this reason we introduce \emph{extended operational Chow group} 
$A^p_{\mr{ext}}(X)$ which contains the operational Chow group. We show that the BGS cycle class map extends to 
$A^p_{\mr{ext}}(X)$. Moreover, if $p=1$ and $X$ has at worst isolated singularity (not necessarily a surface), 
then the extended BGS map is surjective. This further extends the Lefschetz $(1,1)$-theorem to the case of 
isolated singularities.
\end{abstract}

\maketitle

\section{Introduction}

The underlying field will always be $\mbb{C}$.
One of the oldest results in algebraic geometry is the Lefschetz $(1,1)$-theorem.
It states that given a smooth projective variety $X$, every integral $(1,1)$ Hodge class on $X$
is the first Chern class of a line bundle on $X$.
The importance of Lefschetz $(1,1)$-theorem cannot be overstated, it gave rise to the still unsolved,
famous Millennium problem called the Hodge conjecture.
The purpose of this article is to generalize the classical Lefschetz $(1,1)$-theorem to the case of singular varieties.

There have already been various attempts to obtain a Lefschetz $(1,1)$-theorem for singular varieties.
Since there exists an exponential exact sequence for singular varieties, the 
most obvious attempt would be to check if the first Chern class map 
from $\mr{Pic}(X)$ to $H^2(X,\mb{Z})$ maps surjectively to $F^1H^2(X,\mb{C}) \cap H^2(X,\mb{Z})$.
Totaro \cite{totchow} showed that this holds true for all compact toric varieties. 
However, there are examples when the surjectivity fails (see \cite[p. $18$]{totchow}).
Moreover, since $X$ is singular, $H^2(X,\mb{Q})$ is not equipped with a pure Hodge structure.
To consider $(1,1)$ classes, one must restrict to the weight $2$ graded piece $\mr{Gr}^W_2H^2(X,\mb{Q})$. This weight 
graded piece is equipped with a pure Hodge structure. In this article, by \emph{Hodge (1,1)-classes} we will mean 
elements in $H^{1,1}\mr{Gr}^W_2H^2(X,\mb{Q})$. Since all the elements in  $F^1H^2(X,\mb{C}) \cap H^2(X,\mb{\mb{Q}})$
are of weight $2$, it is contained in $H^{1,1}\mr{Gr}^W_2H^2(X,\mb{Q})$. However, the two $\mb{Q}$-vector spaces do not always 
coincide (see \cite[\S $7$]{totchow}). In other words, there are more Hodge $(1,1)$-classes than those  in 
$F^1H^2(X,\mb{C}) \cap H^2(X,\mb{\mb{Q}})$. 

In this article, following Totaro \cite{totchow} we focus on the Bloch-Gillet-Soule (BGS) cycle class map \cite{soub}, instead of the first Chern class map.
Given a complex variety, the BGS cycle class map goes from the operational Chow group $A^p(X)$ (see \cite[Chapter $17$]{fult})
to $\mr{Gr}^W_{2p}H^{2p}(X,\mb{Q})$, the $2p$-th weight graded piece of $H^{2p}(X,\mb{Q})$. 
Note that, $\mr{Gr}^W_{2p}H^{2p}(X,\mb{Q})$ is a subquotient of 
$H^{2p}(X,\mb{Q})$ and there is no 
functorial lift of the BGS map to $H^{2p}(X,\mb{Q})$
(see \cite[\S $7$]{totchow}). The example of Totaro also shows that, in the case $p=1$,
the restriction of the 
BGS cycle class map to $\mr{Pic}(X)$ is not surjective. However, we prove:

\begin{thm}[see Corollary \ref{cor:rat}]
    If $X$ is a normal, projective surface with at worst rational singularity, then the Bloch-Gillet-Soul\'{e} cycle class map:
    \[\mr{cl}_1:A^1(X) \to H^{1,1}\mr{Gr}^W_2H^2(X,\mb{Q})\, \mbox{ is surjective},\]
    where $A^1(X)$ is the operational Chow group.
\end{thm}

See Theorem \ref{thm:rat} for a more general statement.
As a result, the BGS cycle class map for the examples mentioned above (including those where the first Chern class map is not surjective) is surjective (see Example \ref{exa:NS}). 
Unfortunately, the BGS cycle class map is also not always surjective (see Theorem \ref{thm:non-surj}). 
To solve this problem we introduce the notion of \emph{extended operational Chow group}.
Given a projective variety $X$, this is the inverse limit over all resolutions $\wt{X}$ of $X$ of 
a particular subset of $A^p(\wt{X})$ containing $A^p(X)$ (see Definition \ref{defi:ext}). In section \ref{section: lefshetz one,one} we discuss why the extended operational Chow group is a natural object to consider and  is indeed the largest group of cycles which has a natural morphism to the cohomology group of $X$.
Denote by $A^p_{\mr{ext}}(X)$ the \emph{extended operational Chow group of codimension $p$}. 
We show that the extended operational Chow group in codimension $1$ is functorial (see Theorem \ref{thm:lef}).
Furthermore, we prove:

\begin{thm}[see Theorem \ref{thm:lef}]\label{thm:intro-2}
    Suppose that the singular locus $X_{\mr{sing}}$ is of dimension less than $p$. Then,
    \begin{enumerate}
        \item $A^p_{\mr{ext}}(X)$ contains $A^p(X)$ and the BGS cycle class map extends to $A^p_{\mr{ext}}(X)$:       
        \[\mr{cl}_p:A^p_{\mr{ext}}(X) \to \mr{Gr}^W_{2p} H^{2p}(X,\mb{Q})\]
        In particular, the restriction of $\mr{cl}_p$ to $A^p(X)$ is the BGS cycle class map.
        \bigskip
        
        \item {\bf (Lefschetz $(1,1)$ for isolated singularities):} if $p=1$ and $X$ has at worst isolated singularities, 
        then $\mr{cl}_p$ maps surjectively to $H^{1,1} \mr{Gr}^W_2 H^2(X,\mb{Q})$.
    \end{enumerate}
\end{thm}

Finally, we would like to compare our findings with a beautiful result of Arapura \cite{ara-lef}.
It is a standard exercise in hyperenvelopes, to check that there is a natural map from 
the motivic cohomology $H^{2p}_M(X,\mb{Q}(p))$ to the operational Chow group $A^p(X)$. 
Arapura \cite{ara-lef} showed that the composition:
\[H^{2p}_M(X,\mb{Q}(p)) \to A^p(X) \xrightarrow{\mr{cl}_p} H^{p,p}\mr{Gr}^W_{2p}(X,\mb{Q})\]
factors through $F^pH^{2p}(X,\mb{C}) \cap H^{2p}(X,\mb{Q})$. 
Moreover if $p=1$, then the induced map from $H^{2}_M(X,\mb{Q}(1))$ to $F^1H^{2}(X,\mb{C}) \cap H^{2}(X,\mb{Q})$
is surjective. The advantage of this result is that it does not require any condition on the singularity type of $X$.
However, the example of Totaro \cite[\S $7$]{totchow} show that 
$F^1H^{2}(X,\mb{C}) \cap H^{2}(X,\mb{Q})$ is in general \emph{properly} contained in 
$H^{1,1} \mr{Gr}^W_{2}H^{2}(X,\mb{Q})$ i.e., $H^{2}(X,\mb{Q})$ has many more Hodge $(1,1)$-classes 
coming from the singular nature of the variety, than those 
lying in the image of the motivic cohomology group. 
 Theorem \ref{thm:intro-2} fills this gap, in the case of isolated singularities. 
 In particular, we show that if $X$ has at worst isolated singularities, then 
 every Hodge $(1,1)$ class in $H^2(X,\mb{Q})$ arises from an element of $A^1_{\mr{ext}}(X)$.




{\bf{Notation}:} Given a projective variety $X$, we will always denote by $A^*(X)$ the operational Chow group, with rational coefficients.
In the case, when $X$ is smooth and quasi-projective, we denote by $\mr{CH}^*(X)$ the usual Chow group, again with rational coefficients.

 {\emph{Acknowledgements}}:  The present article is motivated by a discussion with Prof. C. Voisin during the workshop ``GLEN workshop - Derived Categories, Hodge theory and singularities". We also thank Prof. M. de Cataldo for his feedback.

\section{Operational Chow group}
In this section we recall the definition and properties of the operational Chow group and the Bloch-Gillet-Soul\'{e} cycle class map.

\subsection{Operational Chow group and the cycle class map}\label{subsecOCG}
Operational Chow groups are generalizations of the classical Chow group.
To define the operational Chow group of an algebraic variety $X$, one considers a 
generalization of the resolution of singularities of $X$, known as an \emph{envelope of} $X$.
Recall, an envelope of $X$ is a proper map $\tau: U \to X$ such that for each $x \in X$, there exists a point $u \in U$
with $\tau(u)=x$ such that $k(u)=k(x)$ (see \cite{fult-rr} for details). 
For example, the envelope of an integral surface $X$
is the disjoint union of the resolution of $X$, resolution of $X_{\mr{sing}}$ (which is a curve or a union of 
points) and the singular locus of $X_{\mr{sing}}$.

Consider now smooth envelopes \[q: U \to X\, \mbox{ and } p: V \to U \times_X U\] of $X$ and the fiber product $U \times_X U$.
Denote by \[p_i: V \xrightarrow{q} U \times_X U \xrightarrow{\pr_i} U,\]
where $\pr_i$ denotes the projection onto the $i$-th coordinate for $i=1,2$.
The operational Chow group $A^*(X)$ is the kernel of the morphism (see \cite[Theorem $\mr{A.3}$]{soub} or {\cite[Theorem 2.3]{kimura}}):
\[p_1^*-p_2^*: \mr{CH}^p(U) \to \mr{CH}^p(V).\]

Denote by 
 \[H^{2p}_{\mr{Hdg}}(X):= \mr{Gr}^W_{2p} H^{2p}(X, \mb{Q}) \cap F^{p} \mr{Gr}^W_{2p} H^{2p}(X, \mb{C}). \]
Using the classical cycle class map over the non-singular algebraic varieties $U$ and $V$, Bloch-Gillet-Soul\'{e} extended this definition to 
singular varieties:

\begin{cor}{\cite[Corollary A.4]{soub}}\label{cor:cyc}
 There is a ``cycle class'' natural transformation of contravariant functors from the category of varieties over $\mb{C}$
 to the category of commutative, graded rings:
  \[ \bigoplus_{p} \mr{cl}_p : \bigoplus_{p}A^{p}(-) \to \bigoplus_{p} H^{2p}_{\mr{Hdg}}(-).\]
\end{cor}

\begin{defi}
For $Y$ a projective variety, the cycle class map mentioned in Corollary \ref{cor:cyc}: 
\[\mr{cl}_p: A^p(Y) \to  H^{2p}_{\mr{Hdg}}(Y,\mb{Q}).\]
will be called the \emph{Bloch-Gillet-Soul\'{e}} cycle class map. It coincides with the usual cycle class map if $Y$ is smooth. 
\end{defi}

\subsection{Application to resolution of singularities}
Let $Y$ be a quasi-projective variety (possibly singular),
of dimension say $n$. Consider a non-singular hyperenvelope of a 
compactification of $Y$ (see \cite[\S $1.4.1$]{soug}
for the definition and basic properties of hyperenvelopes). 
The hyperenvelope gives rise to a 
cochain complex of motives (see \cite[\S $2.1$]{soug}). 
For any positive integer $p$, one 
can then obtain an abelian group $R^i\mr{CH}^p(Y)$
arising as the $i$-th cohomology group after applying the functor 
$\mr{CH}^p(-)$ to the cochain complex of motives (see 
\cite[\S $3.1.4$]{soug}). Observe that $R^i\mr{CH}^p(Y)$
does not depend on the choice of the compactification or the 
hyperenvelope. We recall the following useful result:

\begin{thm}\label{th:chow}
 Fix a positive integer $p$.
 Then, the following holds true for $R^0\mr{CH}^p(Y)$:
 \begin{enumerate}
  \item if $Y$ is projective, then $R^0\mr{CH}^p(Y)$ is 
   the operational Chow group $A^p(Y)$ defined by 
  Fulton and MacPherson (see \cite[Chapter $17$]{fult}),
  \item if $Y$ is non-singular (but not necessarily projective),
  then $A^p(Y)$ is the free abelian group generated by the codimension 
  $p$ subvarieties in $Y$, upto rational equivalence,
  \item if $Y$ is non-singular and $\ov{Y}$ is a compactification 
  of $Y$ with boundary $Z:=\ov{Y} \backslash Y$, we then 
  have the exact sequence:
  \begin{equation}\label{eq:loc}
   0 \to R^0\mr{CH}^p(Y) \to R^0\mr{CH}^p(\ov{Y}) \to 
   R^0 \mr{CH}^p(Z) \to R^1\mr{CH}^p(Y) \to ..
  \end{equation}
  \item if $Y$ is the union of two proper closed subvarieties 
  $Y_1$ and $Y_2$, then we have the exact sequence:
  \begin{equation}\label{eq:may}
   0 \to R^0\mr{CH}^p(Y) \to R^0\mr{CH}^p(Y_1) \oplus R^0\mr{CH}^p(Y_2)
   \to R^0\mr{CH}^p(Y_1 \cap Y_2).
  \end{equation}
 \end{enumerate}
\end{thm}

\begin{proof}
 \begin{enumerate}
  \item This is \cite[Proposition $4$]{soug}.
  \item This is \cite[Proposition $17.3.1$ and 
  Corollary $17.4$]{fult}.
  \item This is \cite[Theorem $2(iii)$ and \S $3.1.1$]{soug}.
  \item This is \cite[Theorem $2(iv)$ and \S $3.1.1$]{soug}.
 \end{enumerate}
\end{proof}

 \begin{cor}\label{cor:ch}
 Let $X$ be an integral, projective variety  and \[\phi: \wt{X} \to X\] be a resolution of singularities and 
 $E$ be the exceptional locus. 
 Then, we have the exact sequence:
 \begin{equation}\label{eq:ch3}
   0 \to   A^p(X) \to A^p(\wt{X}) \oplus A^p(X_{\mr{sing}}) \xrightarrow{\phi^*-i^*} A^p(E).
 \end{equation}
 where $X_{\mr{sing}}$ denotes the singular locus of $Y$ and $i: E \hookrightarrow \wt{X}$ is the inclusion.
 \end{cor}

  \begin{proof}
  Denote by $X_{\mr{sm}}= X \backslash X_{\mr{sing}}$ the smooth locus in $X$. 
  Denote by 
 $A^p_c(Y):=R^0\mr{CH}^p(Y)$.
   By Theorem \ref{th:chow}, we have the following diagram of exact sequences:
   \[\begin{diagram}
       0&\rTo&A^p_c(X_{\mr{sm}})&\rTo^{j_*}&A^p(X)&\rTo&A^p(X_{\mr{sing}})&\rTo^{\partial}&R^1\mr{CH}^p(X_{\mr{sm}})\\
       & &\dTo^{\mr{id}}&\circlearrowleft&\dTo^{\phi^*}&\circlearrowleft&\dTo^{\phi^*}&\circlearrowleft&\dTo^{\mr{id}}\\
       0&\rTo&A^p_c(X_{\mr{sm}})&\rTo&A^p(\wt{X})&\rTo^{i^*}&A^p(E)&\rTo^{\partial}&R^1\mr{CH}^p(X_{\mr{sm}})
   \end{diagram}\]
   This induces the sequence \eqref{eq:ch3}. It remains to 
    check exactness. For injectivity of the first arrow in \eqref{eq:ch3}, observe using the above diagram
    that $\alpha \in A^p(X)$ maps to zero in $A^p(\wt{X}) \oplus A^p(X_{\mr{sing}})$ if and only if 
    there exists $\beta \in A^p_c(X_{\mr{sm}})$ such that $j_*(\beta)=\alpha$ and $\beta$ 
    goes to $0$ in $A^p(\wt{X})$. But, the injectivity 
    of the map from $A^p_c(X_{\mr{sm}})$ to $A^p(\wt{X})$ implies that $\beta$ must be $0$. Therefore, $\alpha=0$. This implies injectivity of the 
    first arrow in \eqref{eq:ch3}. 
    
    We now check the exactness in the middle in \eqref{eq:ch3}.
    Given $\gamma \in A^p(\wt{X})$ and $\beta \in A^p(X_{\mr{sing}})$, 
    if $i^*(\gamma)=\phi^*(\beta)$, then $\partial(i^*(\gamma))=0$. Since the last vertical arrow is 
    just the identity map, this implies $\partial(\beta)=0$. Hence, there exists $\alpha \in A^p(X)$ 
    mapping to $\beta$. This implies $i^*(\gamma-\phi^*(\alpha))=0$.
    By the exactness of the bottom row, there exists $\alpha' \in A^p_c(X_{\mr{sm}})$ 
    mapping to $\gamma-\phi^*(\alpha)$. Taking
    \[\alpha_0:=\alpha+j_*(\alpha'),\, \mbox{ we have } \phi^*(\alpha_0)=\phi^*(\alpha)+\phi^*(j_*(\alpha'))=\phi^*(\alpha)+
    (\gamma-\phi^*(\alpha))=\gamma.\]
    Moreover, the exactness of the top row implies that the image of $\alpha_0$ in $A^p(X_{\mr{sing}})$
    is the same as that of $\alpha$, which is equal to $\beta$.
    In other words, $\alpha_0 \in A^p(X)$ maps to $\gamma \oplus \beta \in A^p(\wt{X}) \oplus A^p(X_{\mr{sing}})$.
    This proves the exactness in the middle and hence the corollary.
  \end{proof}

Combining with the cohomology long exact sequence associated to a resolution, we get the following diagram of exact sequences:

\begin{cor}\label{cor:diag}
Let $X, \wt{X}$ and $E$ be as in Corollary \ref{cor:ch}. If $\dim(X_{\mr{sing}})<p$, then 
we have the following diagram of exact sequences:
  \begin{equation}\label{eq:cyc}
      \begin{diagram}
          0&\rTo&A^p(X)&\rTo^{\phi^*}&A^p(\wt{X})&\rTo^{i_E^*}&A^p(E)\\
          & &\dTo^{\mr{cl}_p}&\circlearrowleft&\dTo^{\mr{cl}_p}&\circlearrowleft&\dTo^{\mr{cl}_p}\\
          0&\rTo&H^{2p}_{\mr{Hdg}}(X)&\rTo^{\phi^*}&H^{2p}_{\mr{Hdg}}(\wt{X})&\rTo^{i_E^*}&H^{2p}_{\mr{Hdg}}(E)
      \end{diagram}
  \end{equation}
\end{cor}

\begin{proof}
Since $p>\dim(X_{\mr{sing}})$, we have $A^p(X_{\mr{sing}})=0$ and $H^{2p}(X_{\mr{sing}},\mb{Q})=0$. 
    The corollary then follows immediately from Corollary \ref{cor:ch} along with the functoriality of the Bloch-Gillet-Soul\'{e} cycle class map (Corollary \ref{cor:cyc}).
\end{proof}

 \section{Lefschetz $(1,1)$-theorem for singular varieties}\label{section: lefshetz one,one}
 Let $X$ be a singular, projective variety. 
 In this section we introduce the notion of \emph{extended operational Chow group $A^p_{\mr{ext}}(X)$} 
 which arises as an inverse limit over
 all resolutions $\wt{X}$ of a certain subset of $A^p(\wt{X})$ containing $A^p(X)$ (see Definition \ref{defi:ext}).  
 We show that this is the largest group of cycles which has a morphism to the cohomology group of $X$. 
 We prove that the Bloch-Gillet-Soul\'{e} cycle class map extend to $A^p_{\mr{ext}}(X)$ and in the case $p=1$, the extended 
 cycle class map is surjective (see Theorem \ref{thm:lef}).

 \subsection{Extended Operational Chow group}
Let $X$ be an irreducible projective variety and $Z \subset X$ be the singular locus. By a \emph{resolution of singularities of} $X$, we mean a
proper, surjective morphism \[\phi:\wt{X}_{\phi} \to X\] such that $\wt{X}_{\phi}$ is smooth and $\phi$ is an isomorphism over $X \backslash Z$.
The \emph{exceptional locus} $E_{\phi} \subset \wt{X}_{\phi}$ \emph{associated to} $\phi$, is the preimage of $Z$ i.e., $E_{\phi}:=\phi^{-1}(Z)$
with the reduced scheme structure.
Denote by $i_{\phi}:E_{\phi} \hookrightarrow \wt{X}_{\phi}$ the natural inclusion and 
\[A^p_{\mr{ext}}(\phi)=\left\{\left. \alpha \in A^p(\wt{X}_{\phi})\, \right\vert\, i_{\phi}^*\left(\mr{cl}_p(\alpha)\right) = 0\,\right\},\]
where $\mr{cl}_p:A^p(\wt{X}_{\phi}) \to H^{2p}(\wt{X}_{\phi},\mb{Q})$ is the cycle class map.

\begin{lem}\label{lem:ext}
    Given a morphism of resolution of singularities of $X$:
    \[\phi':\wt{X}_{\phi'} \to X,\, \phi:\wt{X}_{\phi} \to X\, \mbox{ and } \psi_{\phi',\phi}:\wt{X}_{\phi'} \to \wt{X}_{\phi}\, \mbox{satisfying}\, \phi'=\phi \circ \psi_{\phi',\phi}\]
    the pullback induces a group homomorphism:
    \[\psi^*_{\phi',\phi}:A^p_{\mr{ext}}(\phi) \to A^p_{\mr{ext}}(\phi').\]
\end{lem}

\begin{proof}
For any $\alpha \in A^p_{\mr{ext}}(\phi)$, we know that $\psi^*_{\phi',\phi}(\alpha) \in A^p(\wt{X}_{\phi'})$. We need to prove that, $\psi^*_{\phi',\phi}(\alpha) \in A^p_{\mr{ext}}(\wt{X}_{\phi'})$. 
Using Corollary \ref{cor:diag},  there exists $\alpha_X \in H^{2p}_{\mr{Hdg}}(X)$
such that $\mr{cl}_p(\alpha)=\phi^*(\alpha_X)$ (use the exactness of the bottom row in \eqref{eq:cyc}).
Then, 
\[\mr{cl}_p(\psi^*_{\phi',\phi}(\alpha))=\psi^*_{\phi',\phi}\left(\mr{cl}_p(\alpha)\right)=\psi^*_{\phi',\phi}\left(\phi^*(\alpha_X)\right)=(\phi')^*(\alpha_X).\]
Using Corollary \ref{cor:diag} once again (with the resolution $\phi'$), we conclude that 
\[i_{\phi'}^*\mr{cl}_p(\psi^*_{\phi',\phi}\left(\alpha\right))=i_{\phi'}^*(\phi')^*(\alpha_X)=0,\, \mbox{ where } i_{\phi'}:E_{\phi'} \hookrightarrow \wt{X}_{\phi'}\, \mbox{ is the associated 
exceptional locus}.\]
Hence, $\psi^*_{\phi',\phi}(\alpha) \in A^p_{\mr{ext}}(\phi')$. This proves the lemma.
\end{proof}

 \begin{defi}\label{defi:ext}
 Denote by $\mr{Res}(X)$ the category of resolutions of $X$. 
   By Lemma \ref{lem:ext}, we have a natural functor:
   \[\mc{A}^p_{\mr{ext}}(X):\mr{Res}(X)^{\mr{op}}\, \to \, \mr{Ab}, \, \mbox{ sending } \left(\phi:\wt{X}_{\phi} \to X\right) \mbox{ to } 
   A^p_{\mr{ext}}(\phi).\]
   Note that, $\mr{Res}(X)$ is essentially small and the associated skeleton category $\mr{Res}^{\mr{sk}}(X)$ is small.
   Since the category of abelian groups $\mr{Ab}$ is complete, the inverse limits of $A^p_{\mr{ext}}(\phi)$ exists as $\phi$
   varies over all resolution $\phi \in \left(\mr{Res}^{\mr{sk}}(X)\right)^{\mr{op}}$. As 
   $\mr{Res}^{\mr{sk}}(X)$ is a final subcategory of $\mr{Res}(X)$, 
   \cite[\href{https://stacks.math.columbia.edu/tag/09WN}{Tag 09WN, Lemma $4.17.4$}]{stacks-project}
   implies that
   \[\varprojlim_{\phi \in \mr{Res}^{\mr{sk}}(X)^{\mr{op}}} A^p_{\mr{ext}}(\phi) \cong \varprojlim_{\phi \in \mr{Res}(X)^{\mr{op}}} A^p_{\mr{ext}}(\phi) = \]\[=
   \left\{\left. \left(a_{\phi}\right)_{\phi \in \mr{Res}(X)}\, \right\vert \, a_{\phi} \in A^p_{\mr{ext}}(\phi)\, \mbox{ and } 
   \psi_{\phi',\phi}^*\left(a_{\phi}\right)=a_{\phi'},\, \left(\psi_{\phi',\phi}:\wt{X}_{\phi'} \to \wt{X}_{\phi}\right) \in \mf{M}or\left(\mr{Res}(X)\right)\right\}.\]
   This inverse limit will be called \emph{the extended operational Chow group}, denoted by $A^p_{\mr{ext}}(X)$.
   For consistency, if $X$ is non-singular, we set $A^p_{\mr{ext}}(X):=A^p(X)$ (by our definition of resolution, in this 
   case, the only resolution of $X$ is itself).
 \end{defi}

\begin{rem}
    The difference between the extended operational Chow group and the usual operational Chow group, is that 
    the operation Chow group looks at one hyperenvelope of $X$ which contains many more pieces than the resolution of $X$
    (see \S \ref{subsecOCG}).
    However, the universal property of hyperenvelopes guarantee that the operational Chow group does not depend on the 
    choice of the hyperenvelope. In comparison, for the extended operational Chow group, we take a limit over all resolutions $\wt{X}$ of $X$, the largest subspace in 
    $A^p(\wt{X})$ that contributes to Hodge classes on $X$. This naturally contains the operational Chow group, shown below.
\end{rem}

\subsection{Lefschetz $(1,1)$ and other properties}

In the following theorem we show that the extended operational Chow group contains the usual Chow group and is indeed the largest group with a morphism to the cohomology group of $X$, thus making it the natural candidate for the Lefshetz 
$(1,1)$ property in this setup. We show that indeed the extended cycle class map from the extended 
operational Chow group maps surjectively to the space of $(1,1)$-classes, thereby giving us a Lefschetz $(1,1)$-theorem 
for singular varieties.

 \begin{thm}\label{thm:lef}
 Let $X$ be an irreducible projective variety. Fix an integer $p>0$ such that $\dim(X_{\mr{sing}})<p$. Then, 
 \begin{enumerate}
     \item  the operation Chow group $A^p(X)$ is contained in $A^p_{\mr{ext}}(X)$,
     \item  there is a natural cycle class map
     \begin{equation}\label{eq:cyc-ext}
         \mr{cl}_p:A^p_{\mr{ext}}(X) \to H^{2p}_{\mr{Hdg}}(X) 
     \end{equation}
     which coincides with the Bloch-Gillet-Soul\'{e} cycle class map when restricted to the operational Chow group $A^p(X)$,
     \item ({\bf{Functoriality for $p=1$}):} let $f:Y \to X$ be a morphism of irreducible projective varieties
     such that $f^{-1}(X_{\mr{sm}}) \subset Y_{\mr{sm}}$. Suppose that $X$ and $Y$ has at worst isolated singularities. Then, 
      there is a group homomorphism induced by pullback:
     \[f^*:A_{\mr{ext}}^1(X) \to A_{\mr{ext}}^1(Y). \]
     \item ({\bf{Lefschetz $(1,1)$-theorem}):} If $X$ has at worst isolated singularities, then 
     \[\mr{cl}_1:A^1_{\mr{ext}}(X) \to H^2_{\mr{Hdg}}(X)\, \mbox{ is surjective}.\]
 \end{enumerate}
 \end{thm}

 \begin{proof}
 Let $\phi:\wt{X}_{\phi} \to X$ be a resolution of $X$ and $E$ be the exceptional divisor. Using Corollary \ref{cor:diag},
 we know that $A^p(X)$ is contained in $A^p_{\mr{ext}}(\phi)$ and the cycle class map 
 \[\mr{cl}_p:A^p_{\mr{ext}}(\phi) \hookrightarrow A^p(\wt{X}_{\phi}) \to H^{2p}_{\mr{Hdg}}(\wt{X}_{\phi})\, \mbox{ factors through } 
 H^{2p}_{\mr{Hdg}}(X).\] Given another resolution $\phi':\wt{X}_{\phi'} \to X$ and a morphism (over $X$), $\psi_{\phi',\phi}:\wt{X}_{\phi'} \to \wt{X}_{\phi}$,
  $\psi^*_{\phi',\phi}$ maps the image of $A^p(X)$ in $A^p_{\mr{ext}}(\phi)$ to that in $A^p_{\mr{ext}}(\phi')$
 (see Lemma \ref{lem:ext}) 
 and we have a commutative diagram:
 \[\begin{diagram}
     A^p_{\mr{ext}}(\phi)&\rTo^{\psi^*_{\phi',\phi}}&A^p_{\mr{ext}}(\phi')\\
     &\rdTo_{\mr{cl}_p} &\dTo_{\mr{cl}_p}\\
   & &  H^{2p}_{\mr{Hdg}}(X)
 \end{diagram}\]
  Taking inverse limit, we get a 
 map from $A^p(X)$ to $A^p_{\mr{ext}}(X)$ and a cycle class map 
 (see  \cite[\href{https://stacks.math.columbia.edu/tag/09WN}{Tag 002D, Lemma $4.14.9$}]{stacks-project})
 $\mr{cl}_p:A^p_{\mr{ext}}(X) \to H^{2p}_{\mr{Hdg}}(X)$. 
 Furthermore, since inverse limit is left exact, the map from $A^p(X)$ to $A^p_{\mr{ext}}(X)$ is injective. 
 By contruction, the restriction of $\mr{cl}_p$ to the image of $A^p(X)$ is the same as the Bloch-Gillet-Soul\'{e} cycle 
 class map.  This proves $(1)$ and $(2)$.

  To prove $(3)$,  consider a morphism of resolutions of $X$
  \[\phi:\wt{X}_{\phi} \to X, \phi':\wt{X}_{\phi'} \to X\, \mbox{ and } \psi_{\phi',\phi}:\wt{X}_{\phi'} \to \wt{X}_{\phi}\, \mbox{ satisfying } \phi'=\phi \circ \psi_{\phi',\phi}.\]
  We claim that, $\psi^*_{\phi',\phi}:A^1_{\mr{ext}}(\phi) \to A^1_{\mr{ext}}(\phi')$ is an isomorphism. Indeed, we have the following 
  commutative diagram of short exact sequences (see Corollary \ref{cor:diag}): 
  \begin{equation}\label{eq:pic0}
      \begin{diagram}
      0&\rTo&\mr{Pic}^0(\wt{X}_{\phi})&\rTo&A^1_{\mr{ext}}(\phi)&\rTo^{\mr{cl}_1}&\phi^*\left(H^2_{\mr{Hdg}}(X)\right)&\rTo&0\\
 & & \dTo^{\psi^*_{\phi',\phi}}_{\cong}&\circlearrowleft &\dTo^{\psi^*_{\phi',\phi}}&\circlearrowleft&\dTo^{\psi^*_{\phi',\phi}}_{\cong}\\
         0&\rTo&\mr{Pic}^0(\wt{X}_{\phi'})&\rTo&A^1_{\mr{ext}}(\phi')&\rTo^{\mr{cl}_1}&(\phi')^*\left(H^2_{\mr{Hdg}}(X)\right)&\rTo&0
  \end{diagram}
  \end{equation}
  where the last vertical arrow is an isomorphism because $\phi^*$ and $\left(\phi'\right)^*$ are injective and the first vertical arrow is 
  an isomorphism since $\Pic^0(-)$ is a birational invariant. This implies that the middle vertical arrow is an isomorphism.
  This proves the claim. Moreover, since this isomorphism commutes with composition, we conclude that for each resolution 
  $\phi$ of $X$, the natural morphism 
  \begin{equation}\label{eq:inv-1}
      p_{\phi}: A^1_{\mr{ext}}(X) \to A^1_{\mr{ext}}(\phi)\, \mbox{ is an isomorphism}.
  \end{equation}

  Consider a resolution $\phi:\wt{X}_{\phi} \to X$ and the fiber product $W:=\wt{X}_{\phi} \times_X Y$. 
     Since the preimage of $X_{\mr{sm}}$ is      $Y_{\mr{sm}}$, there exists an irreducible component, say $Y_1$ of $W$
     such that the induced map from 
     $Y_1$ to $Y$ is an isomorphism over $Y_{\mr{sm}}$. Denote by $\wt{Y}$ a resolution of singularities of $Y_1$. Then, $\wt{Y}$ is also 
     a resolution of $Y$. We then have the following commutative diagram:
     \begin{equation}
     \begin{diagram}
    \wt{Y}&\rTo^g &\wt{X}_{\phi}\\
    \dTo^{\phi_Y}&\circlearrowleft&\dTo_{\phi}\\
         Y&\rTo^f &X
     \end{diagram}    
     \end{equation}
     We first observe that the morphism $g$ induces a morphism:
     \[g^*:A^1_{\mr{ext}}(\phi) \to A^1_{\mr{ext}}(\phi_Y).\]
     Indeed, for any $\alpha \in A^p_{\mr{ext}}(\phi)$, we have $i_{\phi}^*(\alpha)=0$, where $i_{\phi}:E_{\phi} \hookrightarrow \wt{X}_{\phi}$
     is the associated exceptional divisor.
    Corollary \ref{cor:diag} implies that there exists $\alpha_X \in H^{2}_{\mr{Hdg}}(X)$ such that $\mr{cl}_1(\alpha)=\phi^*(\alpha_X)$.
     Pulling back these classes to $\wt{Y}$, we get:
     \[\mr{cl}_1(g^*\alpha)=g^*\mr{cl}_1(\alpha)=g^*\phi^*(\alpha_X)=\phi_Y^*f^*(\alpha_X).\]
     Using \eqref{eq:cyc} again, after replacing the resolution of $X$ by the resolution $\phi_Y$ of $Y$, this implies
     $i^*\mr{cl}_p(g^*\alpha)=i^*\phi_Y^*f^*(\alpha_X)=0$, where $i:E \hookrightarrow \wt{Y}$ is the exceptional divisor.
     Hence, $g^*\alpha \in A^1_{\mr{ext}}(\phi_Y)$. 
     This proves our claim that $g^*$ sends $A^1_{\mr{ext}}(\phi)$ to $A^1_{\mr{ext}}(\phi_Y)$.
     
   Using the universal property of inverse limit, it suffices to show that for every resolution 
   $\rho \in \mr{Res}(Y)$, we have a morphism of abelian groups 
   \begin{equation}\label{eq:inv-2}
       q_{\rho}:A^1_{\mr{ext}}(X) \to A^1_{\mr{ext}}(\rho)\, \mbox{ compatible with compositions}.
   \end{equation}
  Note that, for any resolution $\rho:\wt{Y}_{\rho} \to Y$ of $Y$, there exists a resolution $W$ of $Y$ (which is a
  resolution of an irreducible component of the fiber product $\wt{Y} \times_Y \wt{Y}_{\rho}$)
  which dominates both $\wt{Y}_{\rho}$ and $\wt{Y}$. Denote by 
  \[h: W \to \wt{Y}\, \mbox{ and } h_{\rho}:W \to \wt{Y}_{\rho}\, \mbox{ the natural maps}.\]
  Consider the composition:
  \begin{equation}\label{eq:comp-1}
     A^1_{\mr{ext}}(X) \xrightarrow[\sim]{p_{\phi}} A^1_{\mr{ext}}(\phi) \xrightarrow{g^*} A^1_{\mr{ext}}(\phi_Y) \xrightarrow[\sim]{h^*}
  A^1_{\mr{ext}}(\phi_Y \circ h)=A^1_{\mr{ext}}(\rho \circ h_{\rho}) \xrightarrow[\sim]{(h_{\rho}^*)^{-1}} A^1_{\mr{ext}}(\rho) 
  \end{equation}
 where the first isomorphism is \eqref{eq:inv-1} and all other isomorphisms follow from the claim shown above (after replacing the 
 resolutions of $X$ by that of $Y$). Since $\mr{Res}(X)$ and $\mr{Res}(Y)$ are cofiltered categories  (cofiltered 
 by taking resolutions of fiber products) and all the morphisms in \eqref{eq:comp-1} commute with composition, 
 we conclude that \eqref{eq:comp-1} does  not depend on the resolutions $\wt{X}, \wt{Y}$ and $W$. 
 This gives us the required morphism \eqref{eq:inv-2}.
 Using the universal property of inverse limit, we get the morphism $f^*$ from $A_{\mr{ext}}^1(X)$ to $A_{\mr{ext}}^1(Y)$.
This proves $(3)$. 

To prove $(4)$, observe using the diagram \eqref{eq:cyc} that for every resolution $\phi:\wt{X}_{\phi} \to X$, the extended cycle class map 
\eqref{eq:cyc-ext} is surjective in the case when $p=1$.
Indeed, by the Lefschetz $(1,1)$-theorem for every $\alpha \in H^2_{\mr{Hdg}}(X)$, $\phi^*(\alpha) \in H^2_{\mr{Hdg}}(\wt{X}_{\phi})$ is the first Chern class 
of a line bundle. In particular, there exists  $\wt{\alpha}_{\mr{alg}} \in A^1(\wt{X}_{\phi})$ such that 
$\mr{cl}_1(\wt{\alpha}_{\mr{alg}})=\phi^*(\alpha)$. Using the exactness of the bottom row of \eqref{eq:cyc}, we conclude that
$\mr{cl}_1\left(\wt{\alpha}_{\mr{alg}}\right)=\alpha$, where $\mr{cl}_1$ is as in \eqref{eq:cyc-ext}.
This proves the surjectivity of 
$\mr{cl}_1$. Then, the isomorphism \eqref{eq:inv-1} implies that taking the limit over all the resolutions of $X$, the induced cycle class map from 
$A^1_{\mr{ext}}(X)$ to $H^2_{\mr{Hdg}}(X)$ is surjective. This proves $(4)$ and hence the theorem. 
 \end{proof}

\section{When is the B-G-S cycle class map surjective?}

\subsection{Surjectivity of the Bloch-Gillet-Soul\'{e} cycle class map}
 Let $X$ be a projective variety with at worst isolated singularities. 
    Let $\phi: \wt{X} \to X$ be a resolution of singularities  with exceptional divisor $E$ such that every irreducible component of $E$
    is non-singular. Such a resolution exists due to Hironaka. 
    Denote by $E_1,...,E_m$ the irreducible components of $E$.

\begin{thm}\label{thm:rat}
   Notations as above. 
  If the natural map 
  \begin{equation}\label{eq:surj}
      H^1(\mo_{\wt{X}}) \to \bigoplus_{i=1}^m H^1(\mo_{E_i})\, \mbox{ is surjective,}
  \end{equation}
        then the Bloch-Gillet-Soul\'{e} cycle class map \[\mr{cl}_1:A^1(X) \to H^2_{\mr{Hdg}}(X)\, \mbox{ is surjective}.\] 
        Additionally, if $H^1(\mo_{E_i})=0$ for all $1 \le i \le m$, then $A^1_{\mr{ext}}(X)=A^1(X)$.     
\end{thm}

\begin{proof}
Combining Corollary \ref{cor:diag} along with the Mayer-Vietoris sequence,
we have the following diagram of exact sequences induced by pullbacks:
\begin{equation}\label{eq:chow-coh}
    \begin{diagram}
    0&\rTo&A^1(X)&\rTo^{\phi^*}&A^1(\wt{X})&\rTo&\bigoplus_{i=1}^m A^1(E_i)\\
    & &\dTo^{\mr{cl}_X}&\circlearrowleft&\dTo^{\mr{cl}_{\wt{X}}}&\circlearrowleft&\dTo_{\mr{cl}_E}\\
    0&\rTo&H^2_{\mr{Hdg}}(X)&\rTo^{\phi^*}&H^2_{\mr{Hdg}}(\wt{X})&\rTo&\bigoplus_{i=1}^m H^2_{\mr{Hdg}}(E_i)
\end{diagram}
\end{equation}
Since $\wt{X}$ and $E_i$ is non-singular for all $1 \le i \le m$, Lefschetz $(1,1)$-theorem implies that 
the second and the third vertical arrows are surjective with kernels $\mr{Pic}^0(\wt{X})$ and $\oplus \mr{Pic}^0(E_i)$, respectively.
Since $\mr{Pic}^0(-)$ is a quotient of $H^1(\mo_{(-)})$, the surjectivity condition \eqref{eq:surj}
implies that the natural pullback morphism from $\mr{Pic}^0(\wt{X})$ to $\oplus_{i=1}^m \mr{Pic}^0(E_i)$ is surjective. 
By Snake lemma applied to the above diagram, this implies  that the first vertical arrow $\mr{cl}_X$ is surjective. Moreover, 
if $H^1(\mo_{E_i})=0$, then $\mr{Pic}^0(E_i)=0$ for all $1 \le i \le m$. Using Snake lemma once again, we conclude that 
$A^1(X) \cong A^1_{\mr{ext}}(\phi)$. Since $A^1_{\mr{ext}}(X) \cong A^1_{\mr{ext}}(\phi)$ as observed in the proof of Theorem \ref{thm:lef},
this implies that $A^1(X)=A^1_{\mr{ext}}(X)$. This proves the theorem.    
\end{proof}

\begin{cor}\label{cor:rat}
    Let $X$ be a normal, projective surface with at worst rational singularities. Then, 
    $A^1(X)=A^1_{\mr{ext}}(X)$ and the Bloch-Gillet-Soul\'{e} cycle class map from $A^1(X)$ to $H^2_{\mr{Hdg}}(X)$ is surjective.
\end{cor}

\begin{proof}
 If $X$ is a projective surface with isolated rational singularities, there exists a resolution 
 \[\phi:\wt{X} \to X\] such that the exceptional divisor $E$ is the union of non-singular rational curve. 
 Denote by $E_1,...,E_m$ be the irreducible components of $E$. 
 This implies that $H^1(\mo_{E_i})=0$ for all $1 \le i \le m$.
 Then, Theorem \ref{thm:lef} implies that $A^1(X)=A^1_{\mr{ext}}(X)$ and the Bloch-Gillet-Soul\'{e} cycle class map $\mr{cl}_1$ is surjective.
    This proves the corollary.
\end{proof}

 \begin{exa}
     [Totaro]\label{exa:NS}
    Let $C \subset \mb{P}^2$ be a singular, reduced cubic curve and $D$ be a degree $d \ge 4$ curve in $\mb{P}^2$ 
    not containing the singular point of $C$. Denote by $\wt{X}$ the blow-up of $\mb{P}^2$ along the $3d$ (smooth) points on $C$.
    Denote by $C'$ (resp. $D'$) the strict transform of $C$ (resp. $D$) in $\wt{X}$. The line bundle corresponding to the divisor $D'$
    defines a proper, birational morphism 
    \[\phi: \wt{X} \to X,\]
    where $C'$ contracts to a point, say $x \in X$ and $\phi$ is an isomorphism outside $C'$ (see \cite[p. $17$]{totchow}).
    Since $C'$ is a cubic curve, the arithmetic genus of $C'$ is one. As $C'$ is singular, its normalization is of genus $0$.
    Therefore, $X$ has isolated rational singularities. 
    By Corollary \ref{cor:rat}, the Bloch-Gillet-Soul\'{e} cycle class map from $A^1(X)$ to $H^2_{\mr{Hdg}}(X)$ is surjective.
 \end{exa}

\subsection{Non-surjectivity of BGS cycle class map}
Let $C \subset \mb{P}^2$ be a smooth, cubic curve and $S \subset C$ be $10$ distinct points on $C$. Let $\wt{X}$ be the blow-up of $\mb{P}^2$ along $S$.
Denote by $C' \subset \wt{X}$ the strict transform of $C$. Then, $(C')^2=-1$. By Artin's contractibility criterion, $C'$ contracts to a singular point, say $x$:
\[\phi:\wt{X} \to X\, \mbox{ such that } \phi^{-1}(x)=C'.\]

\begin{thm}\label{thm:non-surj}
Let $X$ be as above. Then, for a general choice of $10$ points $S$ in $C$, we have  $A^1(X) \not= A^1_{\mr{ext}}(X)$ and the 
Bloch-Gillet-Soul\'{e} cycle class map from $A^1(X)$ to $H^2_{\mr{Hdg}}(X)$ is not surjective.
\end{thm}

\begin{proof}
Since $S$ is general, there exists two points $p,q \in S$ such that $p-q$ is nontorsion in $\mr{Pic}^0(C)$.
Denote by $E_p, E_q \subset \wt{X}$ the exceptional divisors in $\wt{X}$ corresponding to the blow-ups at $p,q$ respectively.
Consider the invertible sheaf $L:=\mo_{\wt{X}}(E_p-E_q) \in A^1(\wt{X})$. Since $p-q$ is non-torsion, $L|_{C'} \cong \mo_C(p-q)$ is non-torsion, 
therefore defines a non-trivial element in $\mr{Pic}^0(C')$. 
Since $\wt{X}$ is rational (blow-up of $\mb{P}^2$ at finitely many points), $H^1(\mo_{\wt{X}})=0$.
This implies that 
\[\mr{cl}_{\wt{X}}(L) \not= 0\, \mbox{ and } \mr{cl}_E(L|_E)=0, \mbox{ where } \mr{cl}_{\wt{X}}:A^1(\wt{X}) \to H^2_{\mr{Hdg}}(\wt{X})\, \mbox{ and }\, \mr{cl}_E:A^1(E) \to H^2_{\mr{Hdg}}(E).\]
Using this in the diagram \eqref{eq:chow-coh},
we conclude that there exists $\alpha \in H^2_{\mr{Hdg}}(X)$ such that $\phi^*(\alpha)=\mr{cl}_{\wt{X}}(L)$ and $\alpha \not\in \mr{Im}(\mr{cl}_X)$.
Therefore, $L \in A^1_{\mr{ext}}(\phi)$ and does not lie in the image of $A^1(X)$. In particular, 
$A^1(X) \not\cong A^1_{\mr{ext}}(\phi)$ and the Bloch-Gillet-Soul\'{e} cycle class map from $A^1(X)$ to $H^2_{\mr{Hdg}}(X)$ is not surjective.
Moreover, since $A^1_{\mr{ext}}(X) \cong A^1_{\mr{ext}}(\phi)$ (see proof of Theorem \ref{thm:lef}), this implies that 
$A^1(X) \not\cong A^1_{\mr{ext}}(X)$. This proves the theorem.    
\end{proof}

\section{Remarks on higher codimension}
It is not clear whether the extended cycle class map given in Theorem \ref{thm:lef} can be surjective in higher codimension.
 However, there is a similar construction, which gives rise to an analogous surjectivity result.

 \begin{defi}
     Consider the wide subcategory, $\mr{Res}_{\mr{wide}}(X) \subseteq \mr{Res}(X)$ consisting of the 
     same objects as $\mr{Res}(X)$ but a smaller set of morphisms. In particular,    
     a morphism $\psi_{\phi',\phi} \in \mf{M}or(\mr{Res}(X))$ between two resolutions of $X$, 
     \[\phi:\wt{X}_{\phi} \to X,\, \phi':\wt{X}_{\phi'} \to X\, \mbox{ and } \psi_{\phi',\phi}:\wt{X}_{\phi'}  \to \wt{X}_{\phi},\]
     belongs to $\mr{Res}_{\mr{wide}}(X)$ if and only if the induced map between the homologically trivial cycles:
     \[\psi^*_{\phi',\phi}:\mr{CH}_{\mr{hom}}(\wt{X}_{\phi}) \to \mr{CH}_{\mr{hom}}(\wt{X}_{\phi'})\, \mbox{ is surjective}.\]
     An example of such a phenomenon is when $\psi_{\phi',\phi}$ is a composition of blow-ups along smooth centers with 
     vanishing Chow groups in the appropriate degrees (an easy consequence of the blow-up formula). 
     Similar to the extended operational Chow group defined in Definition \ref{defi:ext}, we define a further enlargement:
     \[A^p_{\mr{wide}}(X):=\varprojlim_{\phi \in \mr{Res}_{\mr{wide}}(X)^{\mr{op}}} A^p_{\mr{ext}}(\phi) = \]\[=
   \left\{\left. \left(a_{\phi}\right)_{\phi \in \mr{Res}(X)}\, \right\vert \, a_{\phi} \in A^p_{\mr{ext}}(\phi)\, \mbox{ and } 
   \psi_{\phi',\phi}^*\left(a_{\phi}\right)=a_{\phi'},\, \left(\psi_{\phi',\phi}:\wt{X}_{\phi'} \to \wt{X}_{\phi}\right) \in \mf{M}or\left(\mr{Res}_{\mr{wide}}(X)\right)\right\}.\]
 \end{defi}

\begin{thm}
Let $X$ be a projective variety, $p>0$ an integer and $\dim(X_{\mr{sing}})<p$. Then,
\begin{enumerate}
    \item $A^p(X) \subseteq A^p_{\mr{ext}}(X) \subseteq A^p_{\mr{wide}}(X)$,
    \item the cycle class map $\mr{cl}_p$ mentioned in Theorem \ref{thm:lef} extends to $A^p_{\mr{wide}}(X)$:
    \begin{equation}\label{eq:wide}
        \mr{cl}_p:A^p_{\mr{wide}}(X) \to H^{2p}_{\mr{Hdg}}(X).
    \end{equation}
    \item if the Hodge conjecture holds for smooth, projective varieties, then the cycle class map  \eqref{eq:wide} above
    is surjective for all $p>0$.
\end{enumerate}
\end{thm}

\begin{proof}
The proof of $(1)$ and $(2)$ follows identically as Theorem \ref{thm:lef}$(1)$ and $(2)$.
The proof of $(3)$ follows similarly as Theorem \ref{thm:lef}$(4)$, after a small modification:
in the key commutative diagram \eqref{eq:pic0} replace $A^1(-)$ with $A^p(-)$, $\mr{Pic}^0(-)$ with $\mr{CH}^p_{\mr{hom}}(-)$ and $H^2_{\mr{Hdg}}(X)$ by 
$H^{2p}_{\mr{Hdg}}(X)$. Arguing as in Theorem \ref{thm:lef}, we conclude that for any morphism in $\mf{M}or(\mr{Res}_{\mr{wide}}(X))$:
\[\psi_{\phi',\phi}:\wt{X}_{\phi'} \to \wt{X}_{\phi},\]
we have a surjective morphism between the corresponding operational Chow group:
\begin{equation}\label{eq:surj-1}
    \psi_{\phi',\phi}^*:A^p_{\mr{ext}}(\phi) \twoheadrightarrow A^p_{\mr{ext}}(\phi').
\end{equation}
If the Hodge conjecture holds for all smooth, projective varieties, then for each $\phi \in \mr{Res}(X)$, 
the extended cycle class map from $A^p_{\mr{ext}}(\phi)$ to $H^{2p}_{\mr{Hdg}}(X)$
is surjective. Then, the surjectivity \eqref{eq:surj-1} implies that taking inverse limit induces a surjective morphism:
\[\mr{cl}_p:\varprojlim_{\phi \in \mr{Res}_{\mr{wide}}(X)^{\mr{op}}} A^p_{\mr{ext}}(\phi) \twoheadrightarrow H^{2p}_{\mr{Hdg}}(X).\]
This proves $(3)$ and hence the theorem.    
\end{proof}

\end{document}